\documentclass[reqno,11pt]{amsart}
\usepackage{amsmath,amsthm,amscd,amssymb,graphicx,enumerate,latexsym}

\usepackage[raiselinks,colorlinks]{hyperref}
\hypersetup{citecolor=blue}

\numberwithin{equation}{section}

\theoremstyle{plain}
\newtheorem{thm}{Theorem}[section]
\newtheorem{lemma}[thm]{Lemma}
\newtheorem{cor}[thm]{Corollary}

\theoremstyle{definition}
\newtheorem{defn}{Definition}[section]
\theoremstyle{remark}
\newtheorem{remark}{Remark}[section]

\def\Im{\mathop{\rm Im}\nolimits}
\def\Var{\mathop{\rm Var}\nolimits}

\newcommand{\loc}{\text{\rm{loc}}}

\newcommand{\ess}{\text{\rm{ess}}}

\newcommand{\AC}{\text{\rm{AC}}}

\DeclareMathOperator*{\wlim}{w-lim}

\allowdisplaybreaks

\title[Decaying oscillatory potentials]{A class of Schr\"odinger operators with decaying oscillatory potentials}
\author{Milivoje Lukic}
\date\today
\email{milivoje.lukic@rice.edu}


\begin{document}

\begin{abstract}
We discuss Schr\"odinger operators on a half-line with decaying oscillatory potentials, such as products of an almost periodic function and a decaying function. We provide sufficient conditions for preservation of absolutely continuous spectrum and give bounds on the Hausdorff dimension of the singular part of the spectral measure. We also discuss the analogs for orthogonal polynomials on the real line and unit circle.
\end{abstract}

\maketitle
\section{Introduction}

In this paper, we investigate half-line Schr\"odinger operators
\begin{equation}\label{1.1}
(Hu)(x) = - u''(x) + V(x) u(x),
\end{equation}
with decaying oscillatory potentials $V:(0,\infty) \to \mathbb{R}$. All operators we consider have $0$ as a regular point and are limit point at $+\infty$. Therefore, the expression \eqref{1.1}, together with a choice of boundary condition $\theta\in [0,\pi)$, defines a Schr\"odinger operator $H$ on $L^2(0,+\infty)$, with the domain
\[
D(H) = \{ u \in L^2(0,+\infty) \mid  u, u' \in \AC_\loc,  -u''+ Vu \in L^2, u'(0) \sin \theta = u(0) \cos \theta \}.
\]
The operator $H$ is self-adjoint, and for every $z \in \mathbb{C}$ with $\Im z > 0$, there is a nontrivial solution of $- u''_z + V u_z = z u_z$ which is square-integrable near $+\infty$. This can be used to define the $m$-function
\[
m(z) =  \frac{u'_z(0) \cos\theta + u_z(0) \sin \theta }{u_z(0) \cos\theta - u'_z(0) \sin\theta},
\]
which, in turn, defines a canonical spectral measure $\mu$ by
\[
d \mu =\tfrac 1\pi \wlim_{\epsilon \downarrow 0}  m(x+i\epsilon) dx
\]
(the weak limit is with respect to continuous functions of compact support). The importance of $\mu$ lies in the fact that the operator $H$ is unitarily equivalent to multiplication by $x$ on $L^2(\mathbb{R},d\mu(x))$.

The potentials we consider decay at $+\infty$, so $\sigma_\ess(H)=[0,+\infty)$. The purpose of this paper is to characterize the type of spectrum on $[0,+\infty)$. More precisely, for $E>0$, we study generalized eigenfunctions of $H$, i.e.\ solutions of
\begin{equation}\label{1.2}
- u''(x) + V(x) u(x) = E u(x)
\end{equation}
and estimate the Hausdorff dimension of
\begin{equation}\label{1.3}
S = \left\{ E > 0  \mid \text{not all solutions of \eqref{1.2} are bounded}\right\}.
\end{equation}
The importance of the set $S$, from a spectral theorist's point of view, is that by the work of 
Gilbert--Pearson~\cite{GilbertPearson87}, Behncke~\cite{Behncke91} and Stolz~\cite{Stolz92}, on $(0,+\infty)\setminus S$, $\mu$ is mutually absolutely continuous with the Lebesgue measure.

We denote by $\Var(\gamma,I)$ the variation of the function $\gamma$ on the interval $I$,
\[
\Var(\gamma,I) = \sup_{k\in\mathbb{N}} \sup_{\substack{x_0, \dots,x_k\in I \\ x_0 < \dots < x_k}} \sum_{j=1}^k \lvert \gamma(x_j) - \gamma(x_{j-1})\rvert.
\]

The following is our main result.

\begin{thm}\label{T1.1} Let the potential $V$ be given by
\begin{equation}\label{1.4}
V(x) = \sum_{k=1}^\infty c_k e^{-i\phi_k x} \gamma_k(x),
\end{equation}
where the following conditions hold:
\begin{enumerate}[(i)]
\item (uniformly bounded variation) $\gamma_k(x)$ are functions of bounded variation whose variation is bounded uniformly in $k$, i.e.
\begin{equation}
\sup_k \Var(\gamma_k,(0,\infty)) < \infty;
\end{equation}
\item (uniform $L^p$ condition) for some $p\in \mathbb{Z}$, $p\ge 2$,
\begin{equation}
\int_0^\infty \left( \sup_k \lvert \gamma_k(x)\rvert \right)^p dx < \infty;
\end{equation}
\item (decay of coefficients) for some $\alpha \in (0,\frac 1{p-1})$,
\begin{equation}\label{1.7}
\sum_{k=1}^\infty \lvert c_k \rvert^{\alpha} < \infty.
\end{equation}
\end{enumerate}
Then the set $S$ given by \eqref{1.3} has Hausdorff dimension at most $(p-1)\alpha$, and $[0,\infty)$ is the essential support of the absolutely continuous spectrum of $H$.
\end{thm}

Note that conditions (i)--(ii) above imply that $\lim_{x\to\infty} \gamma_k(x) = 0$ for all $k$.

Bounded variation conditions have been analyzed in spectral theory since Weidmann's theorem~\cite{Weidmann67}, but finite sums of the form \eqref{1.4} were first analyzed by Wong~\cite{Wong09}, in the setting of orthogonal polynomials on the unit circle, in the $L^2$ case. In the Schr\"odinger operator literature, Wigner--von Neumann type potentials  have attracted attention since Wigner--von Neumann~\cite{WignerVonNeumann29} and have been studied by Atkinson~\cite{Atkinson54}, Harris--Lutz~\cite{HarrisLutz75}, Reed--Simon~\cite[Thm XI.67]{ReedSimon3}, Ben-Artzi--Devinatz~\cite{Ben-ArtziDevinatz79} and Janas--Simonov~\cite{JanasSimonov10}. Those results are mostly restricted to the $L^2$ case, with the exception of  Janas--Simonov~\cite{JanasSimonov10} which includes the $L^3$ case.

Theorem~\ref{T1.1} continues our earlier work in \cite{Lukic5}, which was, in turn, the analog of the work \cite{Lukic1} on orthogonal polynomials. It is proved in \cite{Lukic5} that if the potential $V$ is given by a sum of the form \eqref{1.4}, with only finitely many non-zero terms and $V\in L^p$, then $S$ is a subset of an explicit finite set which depends only on $p$ and the set of frequencies $\phi_k$. The constructions of Kr\"uger~\cite{Kruger12} and Lukic~\cite{Lukic5} show that this result is optimal; in particular, when there are finitely many terms, the $p$-dependence of possible singular spectrum is a real phenomenon, and not just an artifact of the method. This encourages us to conjecture that the $p$-dependence of possible Hausdorff dimension in Theorem~\ref{T1.1} is also a real phenomenon; however, no such result is presently known.

In the special case when all the $\gamma_k$ are equal, the potential becomes the product of an almost periodic function and a decaying function.

\begin{cor}\label{C1.2} Let $V(x) = \gamma(x) W(x)$, where the following conditions hold:
\begin{enumerate}[(i)]
\item $\gamma(x)$ has bounded variation;
\item $W(x)$ is an almost periodic function given by
\begin{equation}
W(x) = \sum_{k=1}^\infty c_k e^{-i\phi_k x},
\end{equation}
with \eqref{1.7} satisfied for some $\alpha \in (0,\frac 1{p-1})$;
\item $V \in L^p(0,\infty)$ for some $p\in \mathbb{Z}_+$, $p\ge 2$.
\end{enumerate}
Then the set $S$ given by \eqref{1.3} has Hausdorff dimension at most $(p-1)\alpha$, and $[0,\infty)$ is the essential support of the absolutely continuous spectrum of $H$.
\end{cor}

Corollary \ref{C1.2} is an immediate consequence of Theorem~\ref{T1.1}, except for the observation that the $L^p$ condition can be moved from $V(x)$ to $\gamma(x)$, which is proved later. We singled out this special case because it was the main motivation for our work.  For various classes of functions $W(x)$, multiplied by a decaying $\gamma(x)$, it has been studied which rate of decay preserves a.c.\ spectrum. If, instead of being almost periodic, $W(x)$ was sparse (Pearson \cite{Pearson78}, Kiselev--Last--Simon \cite{KiselevLastSimon98}) or random (Delyon--Simon--Souillard \cite{DelyonSimonSouillard85}, Kotani--Ushiroya \cite{KotaniUshiroya88}, Kiselev--Last--Simon \cite{KiselevLastSimon98}), $L^2$ decay of $V$ would be critical for preservation of a.c.\ spectrum; however, if $W(x)$ was periodic, any decay would suffice to preserve a.c.\ spectrum (Golinskii--Nevai \cite{GolinskiiNevai01}). The answer for almost periodic $W(x)$ has been more elusive; Corollary~\ref{C1.2} gives a partial answer, providing a sufficient condition for preservation of a.c.\ spectrum.

The core of the method is summarized by the following technical lemma. To state the lemma, we need to introduce functions $h_{j}$ of $1+j$ variables, defined recursively by $h_{0}(\eta) = 1$ and
\begin{equation}\label{1.9}
h_{J}(\eta;\phi_1,\dots,\phi_J) = \frac 1{\eta - \phi_1 - \dots - \phi_J} \sum_{j=0}^{J-1} h_{j} (\eta;\phi_{1},\dots,\phi_{j}) h_{J-j-1} (\eta;\phi_{{j+1}},\dots,\phi_{{J-1}})
\end{equation}

\begin{lemma}\label{L1.3}
Let the potential $V$ be given by \eqref{1.4}, and let $\eta\in (0,\infty)$, so that the following conditions hold:
\begin{enumerate}[(i)]
\item (uniformly bounded variation) same as condition (i) of Theorem~\ref{T1.1};
\item (uniform $L^p$ condition) same as condition (ii) of Theorem~\ref{T1.1};
\item (decay of coefficients)
\begin{equation}
\sum_{k=1}^\infty \lvert c_k \rvert < \infty;
\end{equation}
\item (small divisor conditions) for $j=1,\dots,p-1$,
\begin{equation}\label{1.11}
\sum_{k_1,\dots,k_j=1}^\infty \left\lvert  c_{k_1} \dotsm c_{k_j} h_{j}(\eta;\phi_{k_1},\dots,\phi_{k_j}) \right\rvert < \infty.
\end{equation}
\end{enumerate}
Then, for $E=\frac{\eta^2}4$,  all solutions of \eqref{1.2} are bounded.
\end{lemma}

\begin{remark}
The proof of Lemma~\ref{L1.3} shows that for real solutions $u(x)$, the quantity
\begin{equation}\label{1.12}
u'(x)^2 + E u(x)^2
\end{equation}
is bounded as $x\to \infty$, and a simple modification (pointed out in the proof) also shows that \eqref{1.12} converges as $x\to\infty$. However, the solution $u(x)$ does \emph{not}, except in special cases, obey WKB asymptotics in its usual form. This is because for $p>2$, there are correction terms in the Pr\"ufer phase which depend on the frequencies $\phi_j$, and cannot be expressed directly in terms of $V(x)$.
\end{remark}

We also present the analogs of Theorem~\ref{T1.1} for orthogonal polynomials on the real line and unit circle. Their proofs are largely analogous, so we will only explain the necessary modifications. We first state the result for orthogonal polynomials on the real line (OPRL).

\begin{thm}\label{T1.4} Let $\rho$ be a nontrivial probability measure on $\mathbb{R}$ with Lebesgue decomposition $d\rho = f(x) dx + d\rho_s$ into an absolutely continuous and a singular part. Let $\rho$ have diagonal Jacobi coefficients $\{b_n\}_{n=1}^\infty$ and off-diagonal Jacobi coefficients $\{a_n\}_{n=1}^\infty$. 
 
 Assume that there is an integer $p\in \mathbb{Z}$, $p\ge 2$, and a real number $\beta \in (0,\frac 1{p-1})$, such that each of the sequences $\{a_n^2-1\}_{n=1}^\infty$, $\{b_n\}_{n=1}^\infty$ can be written in the form
\begin{equation}\label{1.13}
 \sum_{l=1}^\infty c_l e^{-in\phi_l} \gamma^{(l)}_n,
\end{equation}
such that the following conditions hold:
\begin{enumerate}[(i)]
\item (uniformly bounded variation)
\begin{equation}
\sup_l \sum_{n=1}^\infty \lvert \gamma_{n+1}^{(l)} - \gamma_n^{(l)} \rvert <  \infty;
\end{equation}
\item (uniform $\ell^p$ condition)
\begin{equation}
\sum_{n=1}^\infty \left( \sup_l \lvert \gamma_n^{(l)} \rvert  \right)^p < \infty;
\end{equation}
\item (decay of coefficients)
\begin{equation}
\sum_{k=1}^\infty \lvert c_k \rvert^{\beta} < \infty.
\end{equation}
\end{enumerate}
Then there is a set $S$ of Hausdorff dimension at most $\beta(p-1)$ with $\rho_s( (-2,2) \setminus S)=0$, and $f(x)>0$ for Lebesgue-a.e.\ $x \in (-2,2)$. 
\end{thm}

\begin{remark}
The above theorem assumes that the sequence $\{a_n^2-1\}_{n=1}^\infty$ is of the form \eqref{1.13} and obeys the conditions listed there. The sequence $\{a_n^2-1\}_{n=1}^\infty$ appears naturally in the proof, but for a spectral theorist, it would be more natural to pose conditions on $\{a_n-1\}_{n=1}^\infty$. However, if $a_n-1 =\text{\eqref{1.13}}$, then
\begin{align*}
a_n^2 - 1 & = (a_n-1)^2 + 2(a_n-1) \\
& = \sum_{k,l=1}^\infty c_k c_l e^{-in(\phi_k+\phi_l)} \gamma^{(k)}_n \gamma^{(l)}_n + 2 \sum_{l=1}^\infty c_l e^{-in\phi_l} \gamma^{(l)}_n
\end{align*}
is of the same form (with the same values of $p$ and $\beta$), so there is an immediate corollary where the condition is applied on $\{a_n-1\}_{n=1}^\infty$ instead.
\end{remark}

The next result is for orthogonal polynomials on the unit circle (OPUC).

\begin{thm}\label{T1.5}
Let $\mu$ be a nontrivial probability measure on $\partial\mathbb{D}$ with Lebesgue decomposition $d\mu = w(\theta)\frac{d\theta}{2\pi}+d\mu_s$ into an absolutely continuous and a singular part. Let $\mu$ have Verblunsky coefficients $\{\alpha_n\}_{n=0}^\infty$ of the form
\begin{equation}\label{1.17}
\alpha_n = \sum_{l=1}^\infty c_l e^{-in\phi_l} \gamma^{(l)}_n,
\end{equation}
such that conditions (i)--(iii) of Theorem~\ref{T1.4} hold with some \emph{odd} integer $p\in \mathbb{Z}$, $p\ge 3$, and some $\beta \in (0,\frac 1{p-2})$. Then there is a set $S$ of Hausdorff dimension at most $\beta(p-2)$ with $\mu_s(\partial \mathbb{D} \setminus S)=0$, and $w(\theta)>0$ for Lebesgue-a.e.\ $\theta$.
\end{thm}

\begin{remark}\label{R1.3}
 Note that in the previous theorem, the only critical values of $p$ for the $\ell^p$ condition are $\emph{odd}$ integers. The same phenomenon was noticed for the finite frequency case in \cite{Lukic1}, and is in contrast with orthogonal polynomials on the real line and Schr\"odinger operators, where the statement changes at every integer value of $p$. There is an informal way to understand why this happens. For all systems, the method tells us that critical points are of the form
\begin{equation*}
\eta = \phi_{m_1} + \dots + \phi_{m_k} - ( \phi_{n_1} + \dots + \phi_{n_l} )
\end{equation*}
with $k+l < p$. However, only on the unit circle, we can rotate a measure; by rotating the measure by an angle $\psi$, we shift
\[
\eta \mapsto \eta + \psi, \quad \phi_m \mapsto \phi_m + \psi
\]
so only the critical points with $k-l=1$ are preserved. However, increasing $p$, new points with $k-l=1$, $k+l<p$ emerge only when $p$ exceeds an odd integer value.
\end{remark}

We prove Lemma~\ref{L1.3} in Sections~\ref{S2}--\ref{S3}. Sections~\ref{S4} and \ref{S5} contain proofs of Theorem~\ref{T1.1} and Corollary~\ref{C1.2}, respectively, and Section \ref{S6} describes the adaptations necessary to carry over the method to prove Theorems~\ref{T1.4} and \ref{T1.5}.

\section{Preliminaries}\label{S2}

To analyze solutions of \eqref{1.2}, we use Pr\"ufer variables, first introduced by Pr\"ufer \cite{Prufer26}. For
\begin{equation}
E=\frac{\eta^2}4
\end{equation}
with $\eta>0$ and for a real-valued nonzero solution $u(x)$ of \eqref{1.2},
we define modified Pr\"ufer variables $R(x)$, $\theta(x)$ by
\begin{align}
u'(x) & = \tfrac 12 \eta R(x) \cos (\tfrac 12 \eta x + \theta(x)) \label{2.2} \\
u(x) & = R(x) \sin (\tfrac 12 \eta x+ \theta(x)) \label{2.3}
\end{align}
From \eqref{1.2}, we obtain a system of first-order differential equations for $\log R$ and $\theta$,
\begin{align}
\frac{d\theta}{dx} & =  \frac{V(x)}{\eta} \bigl( \tfrac 12 e^{i[\eta x + 2\theta(x)]} + \tfrac 12 e^{-i[\eta x + 2\theta(x)]} - 1 \bigr)   \label{2.4} \\
\frac{d}{dx} \log R(x) & = \Im \left( \frac {V(x)}{\eta} e^{i[\eta x + 2\theta(x)]} \right) \label{2.5}
\end{align}

Note that, by \eqref{2.2} and \eqref{2.3}, boundedness of $R(x)$ implies boundedness of the corresponding solution of \eqref{1.2}. Thus, the goal becomes to analyze the integral of \eqref{2.5},
\begin{equation}\label{2.6}
\log R(b) - \log R(a) = \Im \int_a^b  \frac {V(x)}{\eta} e^{i[\eta x + 2\theta(x)]} dx.
\end{equation}
Note that we will, indeed, only estimate the imaginary part of the integral in \eqref{2.6}. The real part does not, in general, converge as $b\to \infty$.

Substituting \eqref{1.4} into \eqref{2.6}, our goal becomes to estimate integrals of the form
\begin{equation}\label{2.7}
 \int_a^b  e^{K i[\eta x + 2\theta(x)]} e^{-i(\phi_{m_1} + \dots+\phi_{m_J}) x} \gamma_{m_1}(x) \dots \gamma_{m_J}(x) dx.
\end{equation}
Initially, in \eqref{2.6}, these integrals appear with $K=J=1$, but later in the proof they appear with $J\ge 2$ and $0\le K \le J$.

Integrals of the form \eqref{2.7} can be estimated by the following lemma, which is just a more quantitative version of Lemma~4.1 from~\cite{Lukic5}. To avoid placing an absolute continuity assumption on $\gamma_k(x)$, the proof is expressed in terms of Fubini's theorem rather than integration by parts. Remember that by (i), the variations of the $\gamma_k$ are uniformly bounded,
\begin{equation}\label{2.8}
\tau = \sup_k \Var(\gamma_k,(0,\infty)) < \infty.
\end{equation}

\begin{lemma}\label{L2.1} Let $J, K \in\mathbb{Z}$ with $J\ge 1$ and $0\le K\le J$.  Let $0\le a < b < \infty$, and denote
 \begin{align*}
 \Gamma(x) & = \gamma_{m_1}(x) \dots \gamma_{m_J}(x), \\
\phi & = \phi_{m_1} + \dots + \phi_{m_J}.
 \end{align*}
Then
\begin{equation}\label{2.9}
\left\lvert \int_a^b \left(  (\phi - K \eta) e^{K i[\eta x + 2\theta(x)]} e^{-i\phi x} \Gamma(x) - 2K e^{K i[\eta x + 2\theta(x)]} e^{-i\phi x} \Gamma(x) \frac{d\theta}{dx} \right) dx \right\rvert \le 2  \tau^J.
\end{equation}
\end{lemma}

\begin{proof}
Without loss of generality assume that $\gamma_k$ are left continuous. Then there exist finite positive measures $\nu_k$ on $\mathbb{R}$ and functions $s_k:\mathbb{R} \to \{-1,1\}$ such that $\gamma_k(x) = \int_{[x,\infty)} s_k d\nu_k$ and $\nu_k([x,\infty)) = \Var(\gamma_k,[x,\infty)) \le \tau$ by \eqref{2.8}. Using Fubini--Tonelli's theorem and then integrating in $x$, rewrite the integral on the left-hand side of \eqref{2.9} as
\begin{align*}
\int_a^b  \psi'(x) \Gamma(x) dx & = \int_{[a,\infty)^J} \int_a^{\min(t_1,\dots,t_J)} \psi'(x) s_{m_1}(t_1) \dotsm s_{m_J}(t_J) dx d\nu_{m_1}(t_1) \dotsm d\nu_{m_J}(t_J) \\
& = \int_{[a,\infty)^J} \bigl( \psi(\min(t_1,\dots,t_J)) - \psi(a)\bigr) s_{m_1}(t_1) \dotsm s_{m_J}(t_J)d\nu_{m_1}(t_1) \dotsm d\nu_{m_J}(t_J)
\end{align*}
where $\psi(x) = i e^{i(K\eta-\phi)x} e^{2i K \theta(x)}$. Since $\lvert \psi(x) \rvert = 1$, this implies
\[
\left\lvert \int_a^b  \psi'(x) \Gamma(x) dx \right\rvert \le 2 \int_{[a,\infty)^J}  \lvert s_{m_1}(t_1) \dotsm s_{m_J}(t_J) \rvert  d\nu_{m_1}(t_1) \dotsm d\nu_{m_J}(t_J)
\]
and integrating in $t_1,\dots, t_J$ implies \eqref{2.9}.
\end{proof}

We must keep track of integrals of the form \eqref{2.7} and the multiplicative constants with which they will appear in the method. We need to introduce quite a bit of notation, whose importance will become clear in Section~\ref{S3} (or see \cite{Lukic5} for more motivation). For instance, the integral \eqref{2.7} will appear multiplied by $f_{J,K}(\eta;\phi_{m_1},\dots,\phi_{m_J})$, with a function $f_{J,K}$ which we are about to define.

The functions $f_{J,K}$ and $g_{J,K}$ are introduced in \cite{Lukic5}, for $J,K\in \mathbb{Z}$ with $J\ge 1$ and $0\le K \le J$. For other pairs $(J,K) \in \mathbb{Z}^2$, we take those functions to be zero by convention. They are functions of $1+J$ variables, defined recursively by
\begin{equation}\label{2.10}
f_{1,0}(\eta;\phi_1) = - \frac 1\eta, \qquad f_{1,1}(\eta;\phi_1) = \frac 1\eta,
\end{equation}
and
\begin{align}
g_{J,K}(\eta;\{\phi_j\}_{j=1}^J) & = - \frac{ 2 K  }{K \eta - \sum_{j=1}^J \phi_j } f_{J,K}(\eta;\{\phi_j\}_{j=1}^J), \label{2.11} \\
f_{J,K}(\eta;\{\phi_j\}_{j=1}^J)  & = \frac 1\eta  \sum_{k=K-1}^{K+1} \sum_{\sigma\in S_J} \frac 1{J!} \omega_{K-k} g_{J-1,k}(\eta;\{\phi_{\sigma(j)}\}_{j=1}^{J-1}), \quad J \ge 2, \label{2.12}	
\end{align}
where $S_J$ denotes the symmetric group in $J$ elements and 
\begin{equation}\label{2.13}
\omega_a = \begin{cases} -1 & a =0 \\ \tfrac 1{2} & a = \pm 1 \\ 0 & \lvert a\rvert \ge 2 \end{cases}
\end{equation}
are constants which come from an alternative way of writing \eqref{2.4} as
\begin{equation}\label{2.14}
\frac{d\theta}{dx} =  \frac{V(x)}{\eta} \sum_{a=-1}^1 \omega_a  e^{ia[\eta x + 2\theta(x)]}.
\end{equation}
Notation can be simplified by the following symmetric product.

\begin{defn}
For a function $p_{I}$ of $1+I$ variables and a function $q_{J}$ of $1+J$ variables, their symmetric product is a function $p_{I} \odot q_{J}$ of $1+(I+J)$ variables defined by
\begin{equation*}
 (p_{I}\odot q_{J}) \bigl(\eta; \{\phi_i\}_{i=1}^{I+J} \bigr)  = \frac{1}{(I+J)!} \sum_{\sigma\in S_{I+J} } 
  p_{I} \bigl(\eta; \{\phi_{\sigma(i)}\}_{i=1}^{I} \bigr)  q_{J} \bigl(\eta; \{\phi_{\sigma(i)}\}_{i=I+1}^{I+J}\bigr). \end{equation*}
 \end{defn}

Further, it will be convenient to think of $\omega_a$, with $a\in\mathbb{Z}$, as a function of $1+1$ variables, with values given by \eqref{2.13}, and to introduce $\xi_{J,K}$, for $0\le K\le J$, as a function of $1+J$ variables,
\begin{align}
\xi_{J,K}(\eta;\{\phi_j\}_{j=1}^J) & = \begin{cases} \frac{(-1)^{K-1}}\eta & J=1 \\ 0 & J\ge 2\end{cases}
\end{align}
We can now rewrite \eqref{2.10}, \eqref{2.12} as
\begin{equation} \label{2.16}  
f_{J,K} = \xi_{J,K} + \frac 1\eta \sum_{a=-1}^1 \omega_a \odot g_{J-1,K+a}.
\end{equation}
It will also be useful to have notation for the corresponding functions with flipped signs of all but the first parameter,
\begin{align}
\breve f_{J,K}(\eta; \{\phi_j\}_{j=1}^J) & = f_{J,K}(\eta; \{-\phi_j\}_{j=1}^J), \\
\breve g_{J,K}(\eta; \{\phi_j\}_{j=1}^J) & = g_{J,K}(\eta; \{-\phi_j\}_{j=1}^J),
\end{align}
and for
\begin{equation}
 \mathcal G_{J,0} =  \sum_{j=1}^{J-1}  \sum_{k=1}^{\min\{j,J-j\}} \frac 1{4k} g_{j,k} \odot \breve g_{J-j,k}.
\end{equation}

We now point out some identities among the functions just defined. The importance of these identities is mostly in locating singularities of those functions, rather than in the precise form of the identities. For instance, \eqref{2.11} seems to indicate that $g_{J,K}$ has a singularity when $K\eta = \sum_{j=1}^J \phi_j$, but \eqref{2.21} below implies that many of those singularities are removable and that all non-removable singularities stem from $g_{j,1}$ for some $j\le J$, with $\eta = \sum_{i=1}^j \phi_{m_i}$.

\begin{lemma}\phantomsection \label{L2.2}
\begin{enumerate}[\rm (i)]
\item For $0 \le K \le J$ and $0<k<K$,
\begin{align}
f_{J,K} & = \tfrac 12 \sum_{j=0}^J  f_{j,k} \odot g_{J-j,K-k} \\
g_{J,K} & = \tfrac 12 \sum_{j=0}^J g_{j,k} \odot g_{J-j,K-k} \label{2.21}
\end{align}
\item For $J\ge 2$,
\begin{equation} \label{2.22}
f_{J,0}  - \breve f_{J,0} = (\phi_1+\dots + \phi_J )  \mathcal G_{J,0},
\end{equation}
assuming the parameters $\eta;\phi_1,\dots,\phi_J$ for both sides of the identity;

\item The functions $g_{J,1}$ are just rescaled and symmetrized $h_J$, namely,
\begin{equation}\label{2.23}
g_{J,1}(\eta; \{\phi_j\}_{j=1}^J) =  - \frac 2{\eta^J}  \frac 1{J!} \sum_{\sigma \in S_J} h_{J}(\eta; \{\phi_{\sigma(j)}\}_{j=1}^J).
\end{equation}
\end{enumerate}
\end{lemma}

\begin{proof}
(i) is a rescaled version of \cite[Lemma~5.1(i)]{Lukic5}.

(ii) Start from \eqref{2.11} to note
\[
 \frac 1{2k} (\phi_1+\dots + \phi_J ) g_{j,k} \odot \breve g_{J-j,k} = - f_{j,k} \odot \breve g_{J-j,k} + g_{j,k} \odot \breve f_{J-j,k}.
\]
Summing in $j$ and $k$ and using \eqref{2.16}, we have
\begin{align*}
2(\phi_1+\dots + \phi_J )  \mathcal G_{J,0}
 & = - \xi_{1,1} \odot \breve g_{J-1,1} - \frac 1\eta  \sum_{j=1}^{J-1} \sum_{k=1}^{\min\{j,J-j\}} \sum_{a=-1}^1 \omega_a \odot g_{j-1,k+a}  \odot \breve g_{J-j,k} \\
& \qquad  + \xi_{1,1} \odot g_{J-1,1}  + \frac 1\eta  \sum_{j=1}^{J-1} \sum_{k=1}^{\min\{j,J-j\}} \sum_{a=-1}^1 g_{j,k} \odot  \omega_a \odot \breve g_{J-j-1,k+a} 
\end{align*}
which implies \eqref{2.22} since the triple sums are equal (after a relabeling of indices) and $ f_{J,0} = \tfrac 1\eta  \omega_1 \odot g_{J-1,1} = \tfrac 12 \xi_{1,1} \odot g_{J-1,1}$.

(iii) We prove \eqref{2.23} by induction on $J$. Start by verifying
\[
g_{1,1}(\eta;\phi_1) = - \frac 2\eta \frac 1{\eta - \phi_1} = - \frac 2 \eta h_{1}(\eta;\phi_1).
\]
For $J\ge 2$, by \eqref{2.11}, \eqref{2.16} and \eqref{2.21}, we have
\begin{align*}
g_{J,1}(\eta; \{\phi_j\}_{j=1}^J) & = - \frac{2}{\eta - \sum_{j=1}^J \phi_j} f_{J,1} \\
 & = - \frac{2}{\eta - \sum_{j=1}^J \phi_j} \frac 1\eta  \left( \omega_0 \odot g_{J-1,1} + \tfrac 12 \omega_1 \odot \sum_{j=1}^{J-2} g_{j,1} \odot g_{J-j-1,1} \right).
\end{align*}
By the inductive hypothesis, this implies
\begin{align*}
g_{J,1}(\eta; \{\phi_j\}_{j=1}^J)  & = - \frac{2}{\eta - \sum_{j=1}^J \phi_j} \frac 1{\eta^J}  \left( - 2 \omega_0 \odot h_{J-1} + 2 \omega_1 \odot \sum_{j=1}^{J-2} h_{j} \odot h_{J-j-1} \right).
\end{align*}
Using \eqref{2.13} and $h_0=1$, the inductive step is completed.
\end{proof}

\section{Proof of Lemma~\ref{L1.3}}\label{S3}

In this section, we freely use all assumptions of Lemma~\ref{L1.3}. We break up its proof into several lemmas. Let us start by denoting
\begin{align}
\sigma(x) & = \sup_k \lvert \gamma_k(x)\rvert. \label{3.1}
\end{align}
By assumption (ii), $\sigma \in L^p$.

Denoting
\begin{equation}
\mathcal S_{J,K} (x) =  \sum_{m_1,\dots,m_J=1}^\infty f_{J,K}(\eta; \phi_{m_1},\dots,\phi_{m_J}) \beta_{m_1}(x) \dots \beta_{m_J}(x) e^{iK[\eta x + 2 \theta(x)]},
\end{equation}
where
\begin{equation}\label{3.3}
\beta_k(x) =  c_k e^{-i\phi_k x} \gamma_k(x),
\end{equation}
\eqref{2.6} becomes
\begin{equation}\label{3.4}
\log R(b) - \log R(a) = \Im \int_a^b \mathcal S_{1,1}(x) dx.
\end{equation}
The idea of the proof is to iteratively replace $\mathcal S_{1,1}$ by a sum of $\mathcal S_{J,K}$'s with ever higher values of $J$. We will have to keep track of the errors, so denote
\begin{equation}
E_{J,K} = \sum_{m_1,\dots,m_J=1}^\infty \left\lvert  c_{m_1} \dots c_{m_J}  g_{J,K}(\eta; \phi_{m_1},\dots,\phi_{m_J})\right\rvert
\end{equation}
(note that $E_{J,K}$ is trivially zero unless $1\le K \le J$, since the same is true of $g_{J,K}$) and
\begin{equation}
\mathcal E_{J,0} =  \sum_{m_1,\dots,m_J=1}^\infty \left\lvert  c_{m_1} \dots c_{m_J} \mathcal G_{J,0}(\eta; \phi_{m_1},\dots,\phi_{m_J})\right\rvert
\end{equation}
for $K=0$.

\begin{lemma}\label{L3.1}
$E_{J,K}$ is finite when $1\le K \le J\le p-1$ and $\mathcal E_{J,0}$ is finite for $2\le J \le p$.
\end{lemma}

\begin{proof}
By \eqref{2.23}, since the condition \eqref{1.11} holds for $J=1,\dots,p-1$, $ E_{J,1}$ is finite for the same values of $J$. Now note that \eqref{2.21} implies
\begin{equation}
E_{J,K} \le \tfrac 12 \sum_{j=0}^{J} E_{j,k}  E_{J-j,K-k},
\end{equation}
and \eqref{2.22} implies
\begin{equation}
\mathcal E_{J,0}  \le   \sum_{j=1}^{J-1}  \sum_{k=1}^{\min\{j,J-j\}}  \tfrac 1{4k} E_{j,k} E_{J-j,k},
\end{equation}
and the lemma follows from these two identities.
\end{proof}

\begin{lemma}\label{L3.2}
The sum $\mathcal S_{J,K}(x)$ is absolutely convergent when $0\le K \le J \le p$, and if in addition $J\ge 2$, then
\begin{equation}\label{3.9}
\sum_{m_1,\dots,m_J=1}^\infty \left\lvert f_{J,K}(\eta; \phi_{m_1},\dots,\phi_{m_J}) \beta_{m_1}(x) \dots \beta_{m_J}(x) \right\rvert 
\le  \tfrac 1\eta  \sum_{a=-1}^1 \lvert \omega_a \rvert E_{J-1,K+a} \sum_{l=1}^\infty \lvert c_l\rvert \sigma(x)^J.
\end{equation}
\end{lemma}

\begin{proof}
\eqref{2.16} implies
\[
\lvert f_{J,K} \rvert \le  \lvert \xi_{J,K} \rvert + \tfrac 1\eta   \sum_{a=-1}^1 \lvert \omega_a \rvert \odot \lvert g_{J-1,K+a} \rvert.
\]
Multiplying by
\[
\lvert \beta_{m_1}(x) \dots \beta_{m_J}(x) \rvert \le \lvert c_{m_1} \dots c_{m_J}\rvert \sigma(x)^J
\]
(which follows from \eqref{3.1}) and summing in $m_1,\dots,m_J$ completes the proof.
\end{proof}

\begin{lemma}\label{L3.3}
For $J=1,\dots,p-1$,
\begin{equation}\label{3.10}
\left\lvert  \int_a^b \left( \sum_{K=1}^{J} \mathcal S_{J,K} - \sum_{K=0}^{J+1} \mathcal S_{J+1,K} \right) dx \right\rvert \le \sum_{K=1}^{J}  \tfrac 1K E_{J,K} \tau^J.
\end{equation}
\end{lemma}

\begin{proof}
For $K\ge 1$, use Lemma~\ref{L2.1} and multiply \eqref{2.9} by $\frac 1{2K} g_{J,K}(\eta;\phi_{m_1},\dots,\phi_{m_J})$ to conclude
\[
\left\lvert \int_a^b \left(  f_{J,K} e^{K i[\eta x + 2\theta(x)]} e^{-i\phi x} \Gamma(x) - g_{J,K} e^{K i[\eta x + 2\theta(x)]} e^{-i\phi x} \Gamma(x) \frac{d\theta}{dx} \right) dx \right\rvert \le \frac 1K \lvert g_{J,K} \rvert \tau^J
\]
where we have used \eqref{2.11} and the notation in Lemma~\ref{L2.1}.  Multiply by $c_{m_1} \dots c_{m_J}$, sum in $m_1, \dots, m_J$ from $1$ to $\infty$, and sum in $K$ from $1$ to $J$ to conclude \eqref{3.10}. The sum containing the $g_{J,K}$ turns into the sum of $\mathcal S_{J+1,K}$ by using \eqref{1.4}, \eqref{2.14} and \eqref{2.12}.

The infinite summation is justified by Fubini's theorem, by Lemmas~\ref{L3.1} and \ref{L3.2}.
\end{proof}

\begin{lemma}\label{L3.4} For $J=2,\dots,p$,
\begin{equation}\label{3.11}
\left\lvert \Im  \int_a^b \mathcal S_{J,0}(x) dx \right\rvert \le \mathcal E_{J,0} \tau^J.
\end{equation}
\end{lemma}

\begin{proof}
Without loss of generality, we can assume that for each term \eqref{3.3} in the sum \eqref{1.4}, the sum also contains a term $\bar c_k e^{i\phi_k x} \bar\gamma_k(x)$; we can fulfill this assumption by taking the representation \eqref{1.4} and averaging it with its complex conjugate, since $V(x)$ is real-valued. Then, note that for every term
\[
 f_{J,0}(\eta; \phi_{m_1},\dots,\phi_{m_J}) \beta_{m_1}(x) \dots \beta_{m_J}(x) 
\]
in  $\mathcal S_{J,0}$, there is another term with opposite signs of the $\phi_{m_j}$,
\[
f_{J,0}(\eta; -\phi_{m_1},\dots,-\phi_{m_J}) \bar\beta_{m_1}(x) \dots \bar\beta_{m_J}(x).
\]
Averaging those two terms and using Lemma~\ref{L2.2}(ii) and Lemma~\ref{L2.1}, we can estimate
\begin{align*}
& \;\quad \tfrac 12 \left\lvert  \Im \int_a^b \left(  f_{J,0} \beta_{m_1}(x) \dots \beta_{m_J}(x) + \breve f_{J,0} \bar\beta_{m_1}(x) \dots \bar\beta_{m_J}(x) \right) dx   \right\rvert  \\
& = \tfrac 12 \left\lvert \Im \int_a^b  \left(  (f_{J,0} - \breve f_{J,0} ) \beta_{m_1}(x) \dots \beta_{m_J}(x)  \right) dx \right\rvert \\
& \le  \lvert  \mathcal G_{J,0}(\eta; \phi_{m_1},\dots,\phi_{m_J}) \rvert  \tau^J.
\end{align*}
Summing in $m_1,\dots, m_J$ implies \eqref{3.11}.
\end{proof}

\begin{proof}[Proof of Lemma~\ref{L1.3}]

Summing \eqref{3.10} in $J=1,\dots,p-1$, we obtain
\begin{equation}\label{3.12}
\left\lvert \int_a^b \left( \mathcal S_{1,1}(x) - \sum_{K=1}^p \mathcal S_{p,K}(x)  - \sum_{j=2}^{p} \mathcal S_{j,0}(x) \right) dx  \right\rvert \le \sum_{j=1}^{p-1} \sum_{k=1}^j \frac 1k E_{j,k} \tau^j.
\end{equation}
Meanwhile, using Lemma~\ref{L3.2} for $J=p$, integrating in $x$ and summing in $K$,
\begin{equation}\label{3.13}
\left\lvert \sum_{K=1}^p \int_a^b \mathcal S_{p,K}(x) dx \right\rvert \le \frac 2\eta   \sum_{k=0}^{p-1} E_{p-1,k} \sum_{l=1}^\infty \lvert c_l\rvert \int_a^b \sigma(x)^p dx.
\end{equation}

Taking the imaginary part of \eqref{3.12} and using Lemma~\ref{L3.4} and \eqref{3.13}, we conclude (remembering \eqref{3.4}) that
\begin{equation}\label{3.14}
\lvert \log R(b) - \log R(a) \rvert \le \sum_{j=1}^{p-1} \sum_{k=1}^j \frac 1k E_{j,k} \tau^j + \sum_{J=2}^p \mathcal E_{J,0} \tau^J +  \frac 2\eta \sum_{k=0}^p E_{p-1,k} \sum_{l=1}^\infty \lvert c_l\rvert \int_a^\infty \sigma(x)^p dx.
\end{equation}
All we need from this inequality is that it is an estimate independent of $b$. Thus, $\log R(b)$ is a bounded function as $b\to \infty$, which shows that $u(x)$ is a bounded function.

We presented the proof in this way for (relative) clarity. If we had, instead of using $\tau$, written the estimate in Lemma~\ref{L2.1} in terms of variations of the $\gamma_k$, and used that throughout the method, \eqref{3.14} would be an inequality with a right-hand side that decays to $0$ as $a\to \infty$. This would automatically imply that $\log R(x)$ is Cauchy as $x\to\infty$, and that a finite limit
\[
\lim_{x\to\infty} \log R(x)
\]
exists.
\end{proof}

\section{Proof of Theorem~\ref{T1.1}}\label{S4}

To prove Theorem~\ref{T1.1} starting from Lemma~\ref{L1.3}, we need to estimate the Hausdorff dimension of the set of $\eta$ for which the small divisor condition \eqref{1.11} fails for some $j < p$.  For instance, for $p=2$ we need to estimate the dimension of the set of $\eta$ where
\[
\sum_{l=1}^\infty \left\lvert \frac {c_l}{\eta-\phi_l} \right\rvert =\infty;
\]
for $p=3$ we also need to estimate the dimension of the set where
\[
\sum_{k,l=1}^\infty \left\lvert \frac {c_k c_l}{(\eta-\phi_k)(\eta-\phi_k -\phi_l)} \right\rvert =\infty;
\]
etc. We will use measures with the following property: for $\beta \in [0,1]$, a Borel measure $\nu$ on $\mathbb{R}$ is uniformly $\beta$-H\"older continuous (or U$\beta$H) if there exists $\tilde C<\infty$ such that for every interval $I\subset\mathbb{R}$ with $\lvert I \rvert <1$,
\begin{equation}\label{4.1}
\nu(I) \le \tilde C \lvert I \rvert^\beta,
\end{equation}
where $\lvert \cdot \rvert$ denotes Lebesgue measure. If $\nu$ is finite, the condition $\lvert I \rvert <1$ can be removed (while possibly changing $\tilde C$). The condition \eqref{4.1} enters the proof through the following lemma.

\begin{lemma}\label{L4.1} Let $\nu$ be a finite U$\beta$H measure on $\mathbb{R}$. 
\begin{enumerate}[(i)] 
\item If $\alpha \in (0,\beta)$, then for all $\psi \in\mathbb R$,
\begin{equation}
\int \frac 1{\left\lvert \eta-\psi \right\rvert^{\alpha}} d\nu(\eta) \le D_{\alpha},
\end{equation}
where $D_{\alpha}$ is a finite constant which depends only on $\alpha$ and not on $\psi$;
\item For $J\ge 1$ and $\alpha \in (0,\frac \beta J)$,
\begin{equation}
\int  \left\lvert  h_{J}(\eta;\phi_{1},\dots,\phi_{J}) \right\rvert^\alpha d\nu(\eta)  \le   C_{J} D_{J \alpha},
\end{equation}
where $C_J = \frac 1{J+1} \binom{2J}{J}$ are Catalan numbers.
\end{enumerate}
\end{lemma}

\begin{proof} 
(i) By Fubini's theorem, and picking an arbitrary $\epsilon \in (0,\infty)$,
\begin{align*}
\int \frac 1{\lvert \eta-\psi \rvert^\alpha} d\nu(\eta) & = \int_0^\infty \nu\left(\left\{\eta:  \frac 1{\lvert \eta - \psi \rvert^\alpha} > t \right\}\right) dt \\
& \le \epsilon \nu(\mathbb{R}) + \int_\epsilon^\infty  \tilde C (2 t^{-1/\alpha})^\beta dt \\
& \le  \epsilon \nu(\mathbb{R}) + \tilde C \frac {2^\beta}{\frac \beta\alpha -1} \epsilon^{1-\beta/\alpha}
\end{align*}
which is a bound independent on $\psi$, concluding the proof.

(ii) The proof proceeds by induction. For $J=0$ the statement is trivial.

Assume the statement is true for all $j<J$. Integrating one term of the sum on the right-hand side of \eqref{1.9} and using H\"older's inequality and the inductive hypothesis, we get
\begin{align*}
\int \left\lvert \frac 1{\eta-\phi_1 - \dots - \phi_J} h_j h_{J-j-1} \right\rvert^\alpha d\nu(\eta)  & \le D_{J\alpha}^{1/J} (C_j D_{J\alpha})^{j/J} (C_{J-j-1} D_{J\alpha})^{(J-j-1)/J} \\
& \le C_j C_{J-j-1} D_{J\alpha}.
\end{align*}
Summing in $j$, using \eqref{1.9}, and remembering that Catalan numbers obey the recursion relation
\[
C_J = \sum_{j=0}^{J-1} C_j C_{J-j-1},
\]
we complete the inductive step.
\end{proof}

\begin{lemma}\label{L4.2}
Assume that \eqref{1.7} holds. Then, for a positive integer $j$, the set of $\eta$ for which the condition \eqref{1.11} fails has Hausdorff dimension at most $j\alpha$.
\end{lemma}

\begin{proof}
Denote by $T$ the set of $\eta$ where the condition \eqref{1.11} fails. If the Hausdorff dimension of $T$ was greater than $j\alpha$, then for some $\beta>j\alpha$ we would have $h^\beta(T)=\infty$. Thus, there would exist a subset $T' \subset T$ such that $\nu = \chi_{T'} h^\beta$ is a finite U$\beta$H measure with $\nu(T)>0$ (see, e.g.,\cite[Theorem 5.6]{Falconer86}).

Then Lemma~\ref{L4.1}(ii) implies
\[
\int \sum_{k_1,\dots,k_j=1}^\infty \left\lvert  c_{k_1} \dotsm c_{k_j} h_j(\eta;\phi_{k_1},\dots,\phi_{k_j}) \right\rvert^\alpha d\nu(\eta)  \le   C_{j} D_{j \alpha} \left( \sum_{k=1}^\infty \lvert c_k \rvert^\alpha\right)^j.
\]
Since the integral is finite, the integrand must be $\nu$-a.e.\ finite. However, for $\alpha \in (0,1]$ and a sequence $x_n$ of nonnegative numbers,
\[
\sum_{n=1}^\infty x_n^\alpha < \infty \implies \sum_{n=1}^\infty x_n < \infty;
\]
thus, \eqref{1.11} holds for $\nu$-a.e.\ $\eta$, contradicting $\nu(T)>0$.
\end{proof}

\begin{proof}[Proof of Theorem~\ref{T1.1}]
Conditions (i)--(iii) of Lemma~\ref{L1.3} are trivially satisfied. By Lemma~\ref{L4.2}, the condition \eqref{1.7} holds for all $j=1,\dots,p-1$ away from a set of Hausdorff dimension at most $(p-1)\alpha$. Thus, by Lemma~\ref{L1.3}, the Hausdorff dimension of the set $S$ is at most $(p-1)\alpha$ (the map $\eta \mapsto \frac{\eta^2}4$ obviously preserves Hausdorff dimension).

By the results of Gilbert--Pearson~\cite{GilbertPearson87}, Behncke~\cite{Behncke91} and Stolz~\cite{Stolz92}, boundedness of solutions for $E\in (0,\infty) \setminus S$ implies that the canonical spectral measure $d\mu$ and Lebesgue measure are mutually absolutely continuous on $(0,\infty) \setminus S$, which completes the proof.
\end{proof}

\section{Proof of Corollary~\ref{C1.2}}\label{S5}

Corollary~\ref{C1.2} is a special case of Theorem~\ref{T1.1}, with all the $\gamma_k(x)$ taken to be equal to the same function $\gamma(x)$; by the following lemma, $V(x) \in L^p$ then implies $\gamma(x) \in L^p$, and the corollary is immediate.

\begin{lemma}
Let $W(x)$ be (uniformly) almost periodic and not identically zero, and let $\gamma:(0,\infty) \to \mathbb{R}$ have bounded variation. Let $p\in [1,\infty)$. Then  $W \gamma \in L^p(0,\infty)$ implies $\gamma \in L^p(0,\infty)$.
\end{lemma}

\begin{proof}
If $W$ is almost periodic, then so is $\lvert W\rvert^p$, since the map $t\mapsto \lvert t \rvert^p$ is uniformly continuous on compacts. If $\gamma$ has bounded variation, then so does $\lvert\gamma\rvert^p$, since (by the mean value theorem for $t\mapsto t^p$)
\[
\left\lvert \lvert \gamma(x)\rvert^p-\lvert\gamma(y)\rvert^p\right\rvert \le p \lVert \gamma\rVert_\infty^{p-1} \lvert \gamma(x)-\gamma(y)\rvert.
\]
 Thus, it suffices to prove the lemma for $p=1$.

We may pick $T>0$ for which there exist $\delta, \Delta \in (0,\infty)$ such that for all $a\ge 0$,
\begin{equation}\label{5.1}
\delta \le  \int_{a}^{a+T} \lvert W(x) \rvert dx  \le \Delta.
\end{equation}
The upper bound is trivial with $\Delta = T \lVert W\rVert_\infty$, whereas existence of the lower bound for large enough $T$ is a standard fact for non-zero almost periodic functions (see, e.g., \cite[p. 20]{Besicovitch55}).

For $x, y\in [a,a+T]$, by the triangle inequality,
\[
\lvert \gamma(y) \rvert \le  \lvert \gamma(x) \rvert + \lvert \gamma(x) - \gamma(y) \rvert \le  \lvert \gamma(x) \rvert + \Var(\gamma,[a,a+T]).
\]
Integrating in $y$ from $a$ to $a+T$, we conclude
\[
\frac 1T \int_a^{a+T} \lvert \gamma(y) \rvert dy \le  \lvert \gamma(x) \rvert + \Var(\gamma,[a,a+T]).
\]
Multiplying by $\lvert W(x) \rvert$, integrating in $x$ from $a$ to $a+T$, and using \eqref{5.1}, we obtain
\[
 \frac \delta T \int_a^{a+T} \lvert \gamma(y) \rvert dy  \le \int_{a}^{a+T} \lvert W(x) \gamma(x) \rvert dx +  \Var(\gamma,[a,a+T]) \Delta.
\]
Specialize to $a=nT$ and sum in $n$ to obtain
\[
\frac \delta T \int_0^\infty \lvert \gamma(y) \rvert dy \le  \int_{0}^{\infty} \lvert W(x) \gamma(x) \rvert dx  + \Var(\gamma,[0,\infty))  \Delta,
\]
which completes the proof since the right hand side is finite.
\end{proof}

\section{An outline of the proofs of Theorems~\ref{T1.4} and \ref{T1.5}}\label{S6}

The proofs of Theorems~\ref{T1.4} and \ref{T1.5} follow the same ideas, adapted to the discrete case. We present a discussion of the necessary adaptations, omitting the computational details.

In \cite{Lukic1}, we have developed an iterative scheme for proving theorems similar to Theorems~\ref{T1.4} and \ref{T1.5}, but where \eqref{1.13} and \eqref{1.17} are finite sums. For both OPRL and OPUC, the essential spectrum can be parametrized by $\eta \in [0,2\pi]$, by $\eta\mapsto 2\cos(\eta/2)$ or by $\eta\mapsto e^{i\eta}$, respectively. The suitable analog of Pr\"ufer variables can be presented in a unified way for both OPRL and OPUC (see \cite[Section 4]{Lukic1}), as sequences $r_n, \theta_n$ obeying the recursion relations
\begin{align}
\frac{r_{n+1}}{r_n} & = \frac{ \lvert 1 - \alpha_n  e^{i[(n+1)\eta + 2 \theta_n]}  - c \bar \alpha_n  \rvert }{ \sqrt{ (1 - c \alpha_n)(1   - c \bar \alpha_n) - \alpha_n \bar \alpha_n }  }, \label{6.1}  \\
e^{2 i (\theta_{n+1} - \theta_n)} & = \frac{ 1 - \bar \alpha_n e^{-i[(n+1)\eta + 2 \theta_n]} - c \alpha_n }{1- \alpha_n e^{i[(n+1)\eta + 2 \theta_n]} - c \bar \alpha_n },  \label{6.2}
\end{align}
where
\begin{equation*}
c = \begin{cases}
0 & \text{for OPUC,} \\
1 & \text{for OPRL.}
\end{cases}
\end{equation*}
Here, for OPUC, $\alpha_n$ are just Verblunsky coefficients, whereas for OPRL,
\[
\alpha_n  = \frac{a_n^2 - 1 + e^{i \eta/2} b_{n+1}}{e^{i\eta} - 1}.
\]
Thus, in either case, the sequence $\alpha_n$ is of the form \eqref{1.13}. To discuss boundedness of the sequence $r_n$, we estimate partial sums of \eqref{6.1},
\begin{equation}
\sum_{n=M}^N e^{-in\phi} \Gamma_n e^{i k [(n+1)\eta + 2 \theta_n]},
\end{equation}
where
\begin{align}
\Gamma_n & = \gamma_n^{(k_1)} \cdots \gamma_n^{(k_s)} \bar\gamma_n^{(l_1)} \cdots \bar \gamma_n^{(l_t)}, \\
\phi  & = \phi_{k_1}+\cdots + \phi_{k_s} - \phi_{l_1} - \cdots - \phi_{l_t}, \label{6.5}
\end{align}
and $s+t<K$. The analog of Lemma~\ref{L2.1} becomes (compare with \cite[Lemma 6.1]{Lukic1})

\begin{lemma} With notation as above,
\[
\sum_{n=M}^N \left( (e^{-i(k\eta-\phi)} -1) e^{-in\phi}  \Gamma_n  e^{i k [(n+1)\eta + 2 \theta_n]}   -  e^{-in\phi}  \Gamma_n  e^{ik [(n+1)\eta + 2 \theta_n]}  \bigl( e^{2 i k ( \theta_{n+1} - \theta_n)}  - 1  \bigr) \right) \le 2\tau^{s+t}
\]
where $\tau$ is a uniform bound on the variation of the $\gamma^{(l)}$,
\[
\tau = \sup_l \sum_{n=M}^\infty \lvert \gamma^{(l)}_{n+1} - \gamma^{(l)}_n \rvert.
\]
\end{lemma}


This lemma drives an iterative procedure, and it is proved that the Pr\"ufer amplitude $r_n$ is bounded in $n$ if certain small divisor conditions are met. The singularities involved are of the form
\[
\frac 1{e^{-i(\eta - \phi)} - 1},
\]
with $\phi$ as in \eqref{6.5}, and since these are first order singularities at $\eta \in \phi+2\pi\mathbb{Z}$, they can be handled as in Section~\ref{S4}. For instance, in the $\ell^2$ case, the Pr\"ufer amplitude is bounded if
\[
\sum_{l=1}^\infty \left \lvert  \frac {c_l}{e^{-i(\eta - \phi_l)} - 1}   \right\rvert < \infty.
\]

In the general case, the algebra is more complicated than for Schr\"odinger operators, and one needs to work with functions $f_{I,J,K,L}$, $g_{I,J,K,L}$ parametrized by four indices $I,J,K,L$, as defined in \cite[Section 8]{Lukic1}. The proof needs identities analogous to those in Lemma~\ref{L2.2}. For some of those identities, \cite{Lukic1} avoided finding them explicitly, and instead proved by contradiction that the functions obey desired properties. This indirect proof is easily adapted to the current needs; for instance, if in Section~\ref{S2} we hadn't known that \eqref{2.22} held, but we knew that $f_{J,0}-\breve f_{J,0} = 0$ whenever $\phi_1+\dots + \phi_J=0$, that and the fact that $f_{J,0}-\breve f_{J,0}$ is a rational function would suffice to conclude existence of a rational function $\mathcal G_{J,0}$ such that \eqref{2.22} holds.

A closer look at the algebra shows that for OPRL, we obtain small divisor conditions for integers $j$ with $j<p$, whereas for OPUC, we only obtain small divisor conditions for odd integers $j$ with $j<p$. This explains why Theorems~\ref{T1.4} and \ref{T1.5} give different estimates on the Hausdorff dimension, as was already motivated in Remark~\ref{R1.3}.

\bibliographystyle{amsplain}

\providecommand{\bysame}{\leavevmode\hbox to3em{\hrulefill}\thinspace}
\providecommand{\MR}{\relax\ifhmode\unskip\space\fi MR }
\providecommand{\MRhref}[2]{%
  \href{http://www.ams.org/mathscinet-getitem?mr=#1}{#2}
}
\providecommand{\href}[2]{#2}

\end{document}